\documentclass{amsart}

\usepackage{amssymb,amsmath}
\usepackage{url}
\usepackage{amscd}
\usepackage{texdraw}
\usepackage[all]{xy}
\usepackage{fancyhdr}
\usepackage{amsmath}
\usepackage[all]{xy}

\usepackage{graphicx}
\usepackage{subcaption}

\usepackage{amscd,color}
\usepackage[shortlabels]{enumitem}

\theoremstyle{plain}
\newtheorem{thm}{Theorem}

\newtheorem*{starthm}{Main Theorem}

\newtheorem{prop}[thm]{Proposition}
\newtheorem{lemma}[thm]{Lemma}
\newtheorem{defn}{Definition}
\newtheorem{conjecture}{Conjecture}
\newtheorem{remark}{Remark}[section]

\newcommand{\cala}{{\mathcal A}}

\newcommand{\calf}{{\mathcal F}}

\newcommand{\calk}{{\mathcal K}}
\newcommand{\call}{{\mathcal L}}

\newcommand{\cals}{{\mathcal S}}
\newcommand{\calt}{{\mathcal T}}
\newcommand{\calu}{{\mathcal U}}

\newcommand{\CC}{{\mathbb C}}

\newcommand{\NN}{{\mathbb N}}

\newcommand{\ZZ}{{\mathbb Z}}

\renewcommand{\hat}{\widehat}

\newcommand{\la}{\lambda}

\begin{document}

\title[Ergodicity]{Ergodicity in some families of Nevanlinna Functions}

\author{Tao Chen, Yunping Jiang and Linda Keen}

\address{Tao Chen, Department of Mathematics, Engineering and Computer Science,
Laguardia Community College, CUNY,
31-10 Thomson Ave. Long Island City, NY 11101 and
CUNY Graduate Center, New York, NY 10016
}
\email{tchen@lagcc.cuny.edu}

\address{Yunping Jiang, Department of Mathematics, Queens College of CUNY,
Flushing, NY 11367 and Department of Mathematics and the CUNY Graduate Center, New York, NY 10016}
\email{yunping.jiang@qc.cuny.edu}

\address{Linda Keen, Department of Mathematics, the CUNY Graduate
School, New York, NY 10016}
\email{LKeen@gc.cuny.edu; linda.keenbrezin@gmail.com}

\thanks{This research is partially supported by gifts from the Simons Foundation (\#523341 and \#942077) and PSC-CUNY awards.
It was also supported by the National Science Foundation under Grant No. 1440140, while the third author was in residence at the Mathematical Sciences Research Institute in Berkeley, California, during the spring semester 2022}


\begin{abstract}
We study {\em Nevanlinna functions} $f$ that are transcendental meromorphic functions having $N$ asymptotic values and no critical values. In \cite{KK} it was proved that if the orbits of all the asymptotic values have accumulation sets that are compact and on which $f$ is a repeller, then $f$ acts ergodically on its Julia set.  In this paper, we prove that if some, but not all of the asymptotic values have this property, while the others are prepoles, the same holds true.   This is the first paper to consider this mixed case.

\end{abstract}

\maketitle

\section{Introduction}

An early result of McMullen \cite{Mc} says that if $f$  is a rational map of degree greater than one, and if $P(f)$ is its post-singular set, one of two things holds: either $f$'s Julia set is equal to the whole Riemann sphere and the action of $f$ is ergodic or, it is not the whole sphere and the spherical distance $d(f^n(z),P(f)) \rightarrow 0$ for almost every $z$ in $J(f)$ as $n \to \infty$, that is, the $\omega$-limit set $\omega(z)$ is a subset of $P(f)$ that varies with $z$.   Bock~\cite{Bock2} proved a similar result holds for meromorphic functions.
This  begs the question: what are the conditions on a meromorphic function such that its Julia set is the whole sphere and the meromorphic function acts on the sphere is ergodic or not?   In the realm of entire functions,   Lyubich \cite{Lyu} proved  the exponential function $e^z$, whose Julia set is the sphere, is not ergodic   and Bock  \cite{Bock1}  proved that if the set of singular values of an entire function is  finite, and all of these are pre-periodic but not periodic, then the map is ergodic.
  In the realm of meromorphic functions, Bock, \cite{Bock2}, (see also \cite[Theorem 3.3]{RVS} for a proof of Bock's theorem),  proved that if the ``radial Julia set'', a subset of the Julia set,  has positive measure, the action is ergodic. Other earlier results dealt with the particularly simple example of meromorphic functions with two asymptotic values and no critical values.
 There are partial results on the ergodicity question for this family: Let  $\lambda, \mu \in \mathbb{C}$, and

$$f=\frac{\lambda e^z-\mu e^{-\mu}}{e^z-e^{-z}},$$
where $\lambda, \mu$ are $f$'s two asymptotic values.  Keen and Kotus \cite{KK} have shown that if the accumulation sets of both  $\lambda$ and $\mu$ are compact, and $f$ is a repeller on this set,  then the Julia set is $\widehat{\mathbb{C}}$ and $f$ is ergodic.  By way of contrast, Skorulski \cite{S1,S2} has shown that if there exist natural numbers $p$ and $q$ such that $f^p(\lambda)=f^q(\mu)=\infty$, then the Julia set is $\widehat{\mathbb{C}}$ and $f$ is non-ergodic. (See also \cite{CJK}.)

Weiyuan Qiu asked one of the authors what happens in the remaining case where one asymptotic value lands on a repelling cycle, and the other is a prepole.   In answering his question, we were able to prove a more general result for the full family of functions with finitely many asymptotic values and no critical values, so-called "Nevanlinna functions".  Our main theorem is

\medskip
\begin{starthm}~\label{main}
If $f$ is a Nevanlinna function with $N$ asymptotic values of which $0<K<N$ are prepoles, and if the $\omega$-limit sets of the remaining $N-K$ are compact repellers, then the Julia set is $\widehat{\mathbb{C}}$ and $f$ is ergodic.
\end{starthm}

\begin{remark} Our proof of this theorem implies that for these Nevanlinna functions, the measure of the radial Julia set is positive.
\end{remark}

The case $K=0$ was analyzed in~\cite{KK}.  For the case $K=N$,  we have the following conjecture which we are still working on and will report on in a future paper.

\medskip
\begin{conjecture}~\label{nec}
When $K=N$, the action of $f$ on its Julia  set $\widehat{\mathbb{C}}$ is not ergodic.
\end{conjecture}

\medskip
The proof of our theorem depends on generalizations of some lemmas in \cite{KK}.     After an introductory section in which we give the basic definitions and properties of Nevanlinna functions, we state and prove these lemmas and apply them to the proof of the theorem.

\medskip
\medskip
\noindent {\bf Acknowledgement:} We would like to thank Professor Janina Kotus for her helpful comments and suggestions and  for pointing out Skorulski's papers~\cite{S1,S2} to us.
\medskip
\medskip

\section{Preliminaries}\label{prelim}																																		In this section, we recall some of the basic theory of transcendental meromorphic functions that we will need.    Such a function, $f: \CC \rightarrow \hat{\CC}$ is holomorphic except at the set of poles, $\{ f^{-1}(\infty) \},$ and is a local homeomorphism everywhere except at the set $S_f$ of singular points.   In this paper, we will be interested in those functions for which $\#S_f$ is finite and will assume this throughout.  For such functions, the singular values  are of two types:  \\
Let $v$ be a singular value and let $V$ be a neighborhood of $v$.  Then
\begin{itemize}
\item If, for some component $U$ of $f^{-1}(V)$,  there is a  $u\in U$ such that $f'(u)=0$, then $u$ is   a {\em critical point} and $v=f(u)\in V$ is the corresponding  {\em critical value}, or
\item	 If, for some component $U$ of $f^{-1}(V)$, $f:U \rightarrow V \setminus \{ v \}$ is a universal covering map then $v$ is a {\em logarithmic asymptotic value}.   The component $U$ is called an {\em asymptotic tract} for $v$.   Any path $\gamma(t) \in U$ such that $\lim_{t \to 1} \gamma(t) = \infty$, $\lim_{t \to 1} f(\gamma(t))=v$ is called an {\em asymptotic path} for $v$.
\end{itemize}

At regular or non-singular points, meromorphic functions are local homeomorphisms. The dynamics of meromorphic functions with finitely many singular values have been the focus of many dynamical studies.  In particular, all their asymptotic values are isolated and hence logarithmic.   We, therefore, drop the descriptor logarithmic below and call them asymptotic values.
 
An important tool in studying meromorphic functions with finitely many critical points and finitely many asymptotic values is that they can be characterized by their Schwarzian derivatives.

\begin{defn} If $f(z)$ is a meromorphic function, its {\em Schwarzian derivative} is
$$
S(f) = (\frac{f''}{f'})' - \frac{1}{2}(\frac{f''}{f'})^2.
$$
\end{defn}
The Schwarzian differential operator satisfies the  chain rule condition

$$ S(f\circ g) =S(f) g'^2 +S(g)$$
from which it is easy to deduce that if $f$ is a M\"obius transformation, $S(f)=0$, so that
$f\circ g$ and $g$ have the same Schwarzian derivative.

In \cite{N}, Chap. XI, \S 3, Nevanlinna,  using a technique he calls rational approximation, shows how to, given a finite set of points in the plane and finite or infinite branching data for these points, construct a meromorphic function whose topological covering properties are determined by this data.   The function is defined up to M\"obius transformations.
 He proves

 \medskip
\begin{thm}  The Schwarzian derivative of a meromorphic function with finitely many critical points and finitely many asymptotic values is a rational function.  If there are no critical points, it is a polynomial.  Conversely, if a meromorphic function has a rational Schwarzian derivative, it has finitely many critical points and finitely many asymptotic values. If the Schwarzian derivative is a polynomial of degree $m$, then the meromorphic function has $m+2$ asymptotic values and no critical points.
\end{thm}

In the literature, meromorphic functions with polynomial Schwarzian are often called {\em Nevanlinna functions}  (See e.g. \cite{C,EM}).  These are the focus of this paper.

To prove our results, we will need estimates on the asymptotic behavior of the poles and residues of Nevanlinna functions, summarized in Proposition \ref{derivatzeros} at the end of this section.  These are well known, and there is extensive literature; see, e.g., \cite{H}, Chap. 5 or \cite{L}, Chap. 4, \cite{C}  for details.

We begin by recalling the connection between Nevanlinna functions and
 the second-order differential equation,
\begin{equation}\label{ode}
w'' + P(z) w = 0,
\end{equation}
where $P(z)$ is a polynomial of degree $m$.  The solutions of (\ref{ode}) are holomorphic and form a two-dimensional linear space.    A straightforward calculation shows that if $w_1,w_2$ are linearly independent solutions of (\ref{ode}), then $f=w_1/w_2$ is meromorphic and $S(f)=2P(z)$.

 To develop some intuition for discussing Nevanlinna functions, let us begin with a ``toy" example.
Let $P(z)=1$  so that  equation~(\ref{ode}) becomes
\begin{equation}\label{toy}w''+w=0.
\end{equation}
It is easy to check that its solutions are in the two-dimensional space generated by the ``principal"  solutions $w_1=e^{iz}$ and $w_2=e^{-iz}$ and the quotient  $f$ of these two, or of any pair of linearly independent solutions of equation (\ref{toy}) satisfies $S(f)=2$.

Define the ray  $\rho_0 (t) =\{z=t >0\}$ and for any $R>0$ and $\epsilon\in (0,\pi)$, define the sector $$\cals(R, \epsilon)=\mathcal{S}=\{z:\  |z| >R, |\arg z - \pi |  > \epsilon\}.$$
 Obviously,  $\mathcal{S}\setminus \rho_{0}(t)$ consists of two components, $\calu^+$ containing an infinite segment of the positive imaginary axis
  and $\calu^-$ containing an infinite segment of the negative imaginary axis.
  Note that $w_1$ maps $\calu^+$ to a punctured neighborhood of zero and maps $\calu^-$ to a punctured neighborhood of infinity while $w_2$ interchanges the images.
  Thus, $0$ and $\infty$ are asymptotic values of $w_i$, and these two components are their respective asymptotic tracts.  Note that along the ray  $\rho_0 (t)$ separating these two asymptotic tracts, the function $w_1/w_2$  assumes every value in the unit circle infinitely often.

 Similarly,  for any $\delta >0$ and large $T$, there is an $r$ such that in the neighborhood $\{ z \;|\; |z| > T,  \,  |\arg z |< \delta \}$,
$f=w_1/w_2$ takes every value in $\CC^{*}$ infinitely often. The ray $\rho_{0}(t)$ is called the {\em critical ray} and the argument $\theta_0=0$ of $\rho_{0}(t)$  is called
   the {\em critical direction}.

    For a general Nevanlinna function, we define its {\em critical rays} and {\em critical directions}  as follows:

\medskip
 \begin{defn}  Let $f$ be a Nevanlinna function with Schwarzian derivative $2P(z)$   of degree $m \geq 0$  and  suppose the leading coefficient  of $P(z)$ is $a$.   Set $N=m+2$.  Then, each of the solutions $\theta_k$, $k\in \{0, 1, \cdots, N-1\}$ of the solutions of the congruence
$$
\arg a + N \theta \equiv 0 \mod 2\pi
$$
determines a direction and a ray $\rho_k(t)=\{ te^{\theta_{k} i}\;|\; t>T\}$, at infinity.  Each solution determines a {\em critical direction} and a {\em critical ray} of $f$.
\footnote{ See e.g. \cite{C}, p. 6.  These are often also called   Julia directions and Julia rays of $f$.}
\end{defn}

 We will show that, as in the toy example, there is a sector containing the critical ray $\rho_{k}(t)=\{ te^{\theta_{k} i}\;|\; t>0\}$ for each $0\leq k<N$  on which every solution $f$ of the Schwarzian equation take on infinitely many values infinitely often.

 To do this,  for each $k$,  defined $\mod N$,  we make a change of variable that  essentially  turns the sector  
 $$
 S_{k}= \{\theta_{k-1} < \arg z <\theta_{k+1}\} 
 $$
 of the $z$-plane into a sector of a $Z(z)$ plane on which the transported function acts like the toy example. 
 More precisely, in the rest of this section,  assume  both $R_0 >0 $ and the solution $\theta_k$ are fixed and define
 $$
 Z(z)=\int_{R_0e^{i\theta_k}}^z P(s)^{\frac{1}{2}} ds, \quad z\in S_{k},
 $$
 where the branch of  the square roots is chosen so that after integration, $(az^N)^{1/2}$ is
  real and positive on $\rho_k (t)$. 
   
\begin{lemma}\label{lan4.3.6} For some small $\epsilon_0>0$, the function $Z=Z(z)$ satisfies
$$Z(z) = \frac{2a^{\frac{1}{2}}}{N}z^{\frac{N}{2}}(1 + \rm{o}(1) )\, \text{ as }  z \to \infty, |\arg z-\theta_k|\leq \frac{2\pi}{N}-\epsilon.$$
 Moreover, for any $R>R_0$ and $\epsilon>\epsilon$, $Z(z)$ is univalent on the sector
 $$
 \cals=\{z: \ |z|>R, |\arg z-\theta_k|<2\pi/N-\epsilon\}$$ and $Z$ maps $\cals$ onto a region in the $Z$ plane containing  the sector
  $$\calt=\{Z:\  |Z|>R', |\arg Z -\pi | > \epsilon'\}$$
  where $R'$ is large and $\epsilon'>N\epsilon/2$.
 \end{lemma}

 \begin{proof}(See Lemma 4.3.6 in \cite{L})
 Since  $P(s)^{1/2}=a^{1/2}s^{m/2}(1+o(1))$  for large $|s|$, it follows that
  \begin{equation}\label{Zestimate} 
   Z(z)=\int_{R_0e^{i\theta_k}}^zP(s)^{\frac{1}{2}} ds=\frac{2a^{\frac{1}{2}}}{N}z^{\frac{N}{2}}(1+o(z)) \text{ for large $|z|$.}
  \end{equation}
 Thus, the auxilliary map
   $$\xi=\frac{2a^{\frac{1}{2}}}{N}z^{\frac{N}{2}}$$
   maps the sector
    $$\cals_1=\{z:\ |z|>R_{0}, |\arg z-\theta_k|<\frac{2\pi}{N}-\frac{\epsilon}{2}\}$$
    univalently onto the sector
    $$
    \Sigma'=\{\xi:\ |\xi|>R_{0}', |\arg \xi|<\pi-\frac{N\epsilon}{4}\}
    $$
    in the $\xi$ plane for some large $R_{0}'>0$.
  From equation (\ref{Zestimate}), it follows that $|Z(z)-\xi(z)|=o(|z|^{N/2})$ on $\cals$ and therefore that     $Z(z)$ is univalent on $\cals$,  so that its image in the $Z$ plane contains a sector  of the form $\calt$ as well.
 \end{proof}

 Next,  the
 Liouville transformation,
 $$W(Z)=P(z)^{\frac{1}{4}}w(z),$$ transforms  equation (\ref{ode})  into to a new one for $W(Z)$ as follows:
 \begin{equation}\label{Liou}
 W''(Z)+(1-F(Z) )W(Z)=0, \, \mbox{ where } F(Z)=\frac{1}{4}\frac{P''(z)}{P(z)^2} - \frac{5}{16}\frac{P'(z)^2}{P(z)^3}.
 \end{equation}
 For $R'>>0$ and  $\delta\in(0,\pi)$, let $\calt=\{ Z: \ |Z|>R', |\arg Z -\pi|  >  \delta  \}$ be a sector in the $Z$ plane. On $\calt$,  $F(Z)=\rm{O}(1/Z^2)$ so that asymptotically,    the solutions to   equation~(\ref{Liou})  are asymptotic to the solutions of equation~(\ref{toy}).  The solution space is generated by
 \begin{equation}\label{asypath}
 W_1(Z)= e^{iZ}(1+\rm{O}(1/|Z|)) \mbox{  and } W_2(Z)= e^{-iZ}(1+\rm{O}(1/|Z|)).
 \end{equation}
  Each $W_i(z)$ has two asymptotic values, and the sector $\mathcal{T}$ contains their asymptotic tracts separated by the critical ray: the positive real line.
   Pulling back to the $z-$plane by the map $Z(z)$, we obtain two linearly independent ``principal solutions''
 $$
  w_{i} (z)=P(z)^{-1/4}W_i(Z),  \, \, i=1,2,$$ of the original second order equation, equation~(\ref{ode}),  defined in a sector $\cals_{2}$ of the $z$ plane satisfying,
 $$
 \cals_2 =\{z : \,  | z| >R, \, \, |\arg z - \theta_k | < \frac{2\pi}{N}  - \delta'  \}
 $$
 where $R$ is a large constant and  $\delta'$  is a small constant depending on $\delta$.
  
Using  the asymptotic expressions in equation (\ref{asypath}), we see that
\begin{equation}\label{Ff} F(Z)= \frac{AW_1(Z)+BW_2(Z)}{CW_1(Z)+DW_2(Z)}\sim \frac{Ae^{iZ}+Be^{-iZ}}{Ce^{iZ}+De^{-iZ}}\end{equation}
has two asymptotic values with asymptotic tracts separated by the positive real line.   Since, by  Lemma \ref{lan4.3.6},  $Z(z)$ is univalent,
$$f(z)=F(Z(z))=\frac{Aw_1(z)+Bw_2(z)}{Cw_1(z)+Dw_2(z)}=\frac{AW_1(Z)+BW_2(Z)}{CW_1(Z)+DW_2(Z)}$$ has two asymptotic values with asymptotic tracts separated by the critical  ray $\rho_{k} (t)$ in the sector $\cals_2$.

\begin{remark} Since there are $N$ solutions to the congruence, there are $N$ possible choices for $\theta_k$.  Applying the above transformation theory to each defines a sector in the $z$-plane containing a central critical ray and bounded by adjacent rays.   The pullback solutions for each have two asymptotic values with asymptotic tracts in the complement of the critical ray. Pairs of adjacent sectors overlap on one asymptotic tract.  Thus, $f$ has $N$ asymptotic values.
\end{remark}

Equation~(\ref{asypath}) also shows that the   $w_i$ have no zeros in the sectors where they are defined but that, for any $A, B\in \mathbb{C}^*$, the equation  $Aw_1+Bw_2=0$ has infinitely many zeros. We next show that these zeroes accumulate along the critical rays. 

\medskip
   \begin{prop}\label{polesdist}  Let $w_1, w_2$ be the principal solutions defined in the sector $\cals$ containing the critical ray  $\rho_{k} (t)$.  If   $A,B$ are non-zero constants and    $w=Aw_1 +  Bw_2$ then $w=0$ has infinitely many solutions $s_j$ in $\cals$.    Label them so that $\cdots \le |s_j| \leq |s_{j+1} |\le \cdots$.  Then, $ |s_j| \sim  \text{O}(|j|^{2/N})$ and  $\lim_{j\to \infty} |\arg s_j - \theta_k | =0$, \end{prop}

\begin{proof}(See also Page $64$ in \cite{L})
Set
 $G(z)=\frac{1}{2i} \log \frac{W_1}{W_2}$. By lemma~\ref{lan4.3.6},  if   $z \in  \cals$ is sufficiently large,
 $$  G(z)=Z(z)+\text{o}(1)= \frac{2a^{\frac{1}{2}}}{N} z^{\frac{N}{2}}(1 + \text{o}(1)) . $$
 Furthermore, the zeroes $s_j$ of $w$,  satisfy
 $$2iG(s_j) = \log(-\frac{B}{A}) +2 j \pi i, \text{ or } G(s_j) = \frac{1}{2i}\log( -\frac{B}{A}) + j \pi;$$ that is $G(s_j)$ lies near the positive line.

Combining these, we have,
 $$\frac{2a^{\frac{1}{2}}}{N} s_j^{\frac{N}{2}}(1 + \text{o}(1))=\frac{1}{2i}\log(-\frac{B}{A}) + j \pi.$$
 Hence as $j \to \infty$, $\arg s_j \sim \theta_k$;  and there is a constant $c_1$ such that $s_j \sim c_1 |j|^{2/N}$.
 \end{proof}
 
For any function 
\begin{equation}\label{soln}  f(z)=\frac{aw_1 +bw_2}{cw_1+dw_2}
\end{equation}
 its poles are zeros of $cw_1+dw_2$, and the estimate of the residue at each pole can be computed using the definition:   $Res(f,s_j) = \lim_{z \to s_j} (z-s_j)f(z)$.  
 
 \begin{prop}[Page $6$ in \cite{C}]\label{residue}  Let $f$ be as in equation~(\ref{soln}).  Denote the poles of $f$ in $S$ and their respective residues by $s_j$ and $r_j$, and assume the poles are labeled so that $\cdots \le | s_j| \le |s_{j+1}| \le \cdots $. Then
$$r_j = \frac{1}{2i} \big( \frac{a}{c}- \frac{b}{d}\big) P(s_j)^{-\frac{1}{2}} \sim c_2 \cdot s_j^{-\frac{N-2}{2}}, $$
for some constant $c_2$.
\end{prop}

 From the two propositions above,  it follows that the relation between the residues and the poles is
 \begin{equation}\label{pole-res} | r_j | \sim c_2|s_j|^{-\frac{N-2}{2}} \sim c_3 |j|^{-\frac{N-2}{N}}.
 \end{equation}

\medskip
\begin{prop}\label{derivatzeros} As above,  let $f$ be a solution of $S(f)=2P$ where $P$ is a polynomial of degree $m$.  Set $N=m+2$.  For any $z_0\in \mathbb{C}$, denote its preimages  in a given sector $S$  by $p_j$ and label them so that $\cdots \le | p_j| \le |p_{j+1}| \le \cdots $. Then there exists a constant $c_4>0$ such that $|f'(p_j)|\sim c_4j|^{(N-2)/N}$.
\end{prop}

\begin{proof}
Let $$g(z)=\frac{1}{f(z)-z_0};$$ then $S(g)=S(f)=2P$.   If $p_j$ are solutions of $f(z)=z_0$, they are poles of $g(z)$ so that  proposition~\ref{residue} implies that $|\text{res}(g, p_j)| \sim c_3|j|^{-(N-2)/N}$. Furthermore, since $g(z)=1/(f(z)-z_0)$,  a simple computation shows that  $$|f'(p_j)|=\frac{1}{|\text{res}(g, p_j)|}\sim c_4|j|^{\frac{N-2}{N}}$$
\end{proof}

In the proofs of our results  we will repeatedly use the Koebe distortion theorems to obtain estimates on the behavior of the Nevanlinna functions at regular points.  Many proofs exist in the standard literature on conformal mapping.  (See e.g. \cite{Ahlfors}, Theorem 5.3.) For the reader's convenience, we state the theorems here without proof. 

\begin{thm}[Koebe Distortion Theorem]
Let $f: D(z_0, r)\to \mathbb{C}$ be a univalent function, then for any $\eta<1$,
\begin{enumerate}
\item $\displaystyle |f'(z_0)|\frac{\eta r}{(1+\eta)^2}\leq |f(z)-f(z_0)|\leq |f'(z_0)|\frac{\eta r}{(1-\eta)^2}$, $z\in D(z_0, \eta r)$,
\item If $\displaystyle T(\eta)=\frac{(1+\eta)^4}{(1-\eta)^4}$, $\displaystyle \frac{|f'(z)|}{|f'(w)|}\leq T(\eta)$, for any $z, w\in D(z_0, \eta r)$.
\end{enumerate}
\end{thm}

\begin{thm}[Koebe 1/4 Theorem]
Let $ f: D(z_0, r)\to \mathbb{C}$ be a univalent function, then $$D \Big(f(z_0),\frac{r|f'(0)|}{4} \Big)\subset f(D(z_0, r)). $$
\end{thm}

\section{The Main Theorem}
Let $\calf_N$ be the set of Nevanlinna functions with $N$ asymptotic values.  For $f  \in \calf_N$ and $i=1, \ldots, N$  denote the  asymptotic values by $ \la_i$ and the corresponding  asymptotic tracts  by $T_i$.    Assume there is an integer $K$, $1 \leq K < N $, and  integers $p_i \geq 0$, $i=1, \ldots, K$,    such that
 $$ f^{p_i}(\la_i)=\infty, \,\, i=1, \ldots,K.$$
If $\la_i=\infty$, $p_i=0$.  This can happen for at most $N/2$ asymptotic values and the asymptotic tracts of these infinite asymptotic values must be separated by the asymptotic tract of a finite asymptotic value.
Also assume that for each $i=K+1, \cdots, N$, the accumulation set, $\omega(\la_i)$, of the orbit of $\lambda_i$, is a compact repeller;  that is, there exists a $\kappa>1$ such that for each $z \in \omega(\la_i)$, there exists an $n=n(z)$, such that $|(f^{n})'(z)|>\kappa$.  Note that this implies that these asymptotic values are finite.

Define 
$$
I=I(f)=\{z\in \mathbb{C}\ | \ f^n(z)\to \infty \},
$$
and  
$$
L=L(f)=\{z\in \mathbb{C}\ | \ \omega(z)= \cup_{i=K+1}^{N}\omega(\lambda_i)\}.
$$
The proof of the main theorem depends on the following  theorems:
\medskip
\begin{thm}~\label{lem1}
The set $I$ is of measure zero.
\end{thm}


\medskip
\begin{thm}~\label{lem2}
 The set $L$ is of measure zero.
\end{thm}

Versions of these theorems are proved in \cite{KK} under the assumption that  all of the asymptotic values accumulate on a compact repeller.

\subsection{The measure of the set $I$}

 For each $1\leq i \leq K$,
denote the orbit of the prepole asymptotic value $\la_i$ by

$$
Orb(\lambda_i)=\{\lambda_i, f(\lambda_i), \cdots, f^{p_i-1}(\lambda_i),\infty\}.
$$
If $\la_i=\infty$ for some $i$,  $Orb(\la_i)= \{ \infty \}$.

Let $S=\{1,2,\cdots, K\}$.  Since there are $2^K-1$  distinct non-empty subsets of $S$,   label them $S_l$, $l=1, \ldots,  2^K-1$ and denote the collection $\Sigma$.  For any  $S_l $,  define
$$Orb_l=\cup_{i\in S_l} Orb(\lambda_i).$$

For $S_l\in \Sigma$, where $l=1, \ldots, 2^K-1$, define
$$
I_l=I_l(f)=\{z\in \mathbb{C}\ | \ \omega(z)= Orb_l \}.
$$

\begin{thm}\label{main1}
 Each of the sets $I_l$ is of measure zero.
 \end{thm}

The proof of this theorem depends on the next two lemmas.

\medskip
Fix $R>>0$, and let $\cala_R=\{z\in \CC, |z|>R\}$; 
For each $1\leq i\leq K$,  and $ j \in \ZZ$, let $ b_{ij}=f^{-1}(\la_i)$.

Because $\lambda_i$ is prepole of order $p_i$,   one component of $f^{-p_i}(\cala_R)$ is the  topological disk  $D_i$ punctured at $\lambda_i$.   Therefore the set of components of $f^{-1}(D_i)$ consists of the asymptotic tract $T_i$ of $\lambda_i$ and the topological disks $V_{ij}$ punctured at $b_{ij}$.
For each $S_l\in \Sigma$ and $z \in \cup_{i\in S_l}(\cup_{j} V_{ij}\cup T_i)$, define the map $\sigma_l(z)=f^{p_i+1}(z)$.

%

\begin{lemma}\label{lem3a} If $z\in \cup_{i=1}^K T_i$, then $$|\sigma_l'(z)|>\frac{|\log |\sigma_l(z)|-\log R|}{4\pi}\cdot \frac{|\sigma_l(z)|}{|z|}$$
\end{lemma}

\begin{proof} Since $f^{p_i}: D_i\to A_R$ is conformal and $f: T_i\to D_i$ is a universal covering, it follows that  $\sigma_l: T_i\to A_R$ is also a  universal covering.  The rest of proof given below follows along the lines of the corresponding proof in \cite{Lyu}.

Consider $ H_R=\log \cala_R$, the right half plane with real part greater than $R$ and
 let $\calu_l = \log (\cup_{i \in S_l}T_i)$.  Then   $\calu_l \subset H_R$ and consists of infinitely many disjoint simply connected components $U_{im}$, $i \in S_l, m \in \ZZ$; moreover,  there is an $\epsilon_{im}>0$, depending on $R$ such that each  $U_{im}$ is fully contained inside a strip of height $2\pi -\epsilon_{im}$.  Because there are at most $K$ sets $U_{im}$,  the sum of their heights   is less than  $(2\pi K/N-\epsilon_R)$ where $\epsilon_R=\sum \epsilon_{im}$ depends on $R$ and $S_l$.

For each $U_{im}$ there is a conformal map $F_{im}$ to $H_R$ such that the following diagram commutes:

   \medskip
\medskip
\centerline{
\xymatrix{
U_{im} \ar[d]^{\exp} \ar[r]^{F_{im}} &H_R\ar[d]^{\exp}\\
T_i  \ar[r]^{\sigma_l} &\cala_R}.}

 For each point $z_0\in \cup T_i$, denote the lifts of $z_0$ and $\sigma_l(z_0)$ by $w_0\in U_{im}$ and $w_1\in H_R$ respectively. Note that $\Re w_1=\log |\sigma_l(z)|$.   Consider $D=D(w_1, \Re w_1-\log R)$ and its preimage under $F_{im}$. By the $1/4$ Koebe theorem,  its preimage contains a disk of radius
  $$\frac{\Re w_1-\log R}{4|F'_{im}(w)|}. $$
  As the width of each strip is less than $2\pi$, $$|F'_{im}(w)|\geq \frac{\Re w_1-\log R}{4\pi}.$$ The lemma now follows from the chain rule.

 \end{proof}

Theorem~\ref{lem1} follows directly since $ I \subset \cup_{S_l \in \Sigma} I_l$.

The next lemma  is the analog of   lemma~\ref{lem3a} in the case that  $S_l\in \Sigma$ and $z \in \cup_{i\in S_l}(\cup_{j} V_{i{j}})$.
Fix $i\in S_l$ and, suppressing the index $i$ for readability,  denote  the zeros of  $f(z) -\la_i=0$ by $b_j$.  Note that by proposition~\ref{derivatzeros},  $|f'(b_j)|\sim c|j|^{(N-2)/N}$.

\begin{lemma}\label{lem3b}

There exists a neighborhood  $V_j'$ and of $b_j$ and a constant $b>0$ such that $ V'_j\subset \overline{V'_j}\subset V_j$, $$V'_j\subset D(b_j,\frac{b}{|j|^{\frac{N-2}{N}}R}),$$ and for $z \in U_j$ and  for some constant $ B>0$,
$$|\sigma_l'(z)|>BR|j|^{\frac{N-2}{N}}$$
\end{lemma}

\begin{proof}

For each $\lambda_i\in S_l$, denote the pole $f^{p_i-1}(\lambda_i)$ by $s_i$. Then expanding $f$ at $s_i$,
 $$f(z)=\frac{r_i}{z-s_i}(1+\phi_i(z))$$
 where $r_i$ is the residue of $f$ at $s_i$, and $\phi_i$ is analytic at $s_i$. Consider the annular region $\mathcal{A}_{2R}\subset \mathcal{A}_R$, and denote by $g$ the  branch of $f^{-1}$ such that   $g(\mathcal{A}_R)$ is a punctured neighborhood of $s_i$.  Set $h=g(1/z): D(0, 1/R)\to \mathbb{C}$, so that $0$ is a removable singularity and  let $U=h(D(0, 1/2R))=g(\mathcal{A}_{2R}) \cup \{\infty \}$. Then $h$ is conformal  and $h'(0)=r_i$.  The Koebe distortion theorem applied to $h$  proves that  for any $z\in D(0,1/2R)$,
 $$|h(z)-h(0)|\leq |h'(0)|\frac{\frac{1}{2} \cdot \frac{1}{R}}{(1-\frac{1}{2})^2}=\frac{2r_i}{R}.$$
 The Koebe $1/4$ theorem applied to $h$ on $D(0, 1/2R)$ proves that $D(s_i, r_i/8R)\subset U$.  Combining these gives
 $$D(s_i,\frac{|r_i|}{8R})\subset U\subset D(s_i,\frac{2|r_i|}{R}).$$
 Therefore for any $z\in U$,
 \begin{equation}\label{deriatinfinity}|f'(z)|=|-\frac{r_i(1+\phi_i(z))}{(z-s_i)^2}+\frac{r_i\phi'_i(z)}{z-s_i}|\geq |\frac{r_i}{z-s_i}|\geq \frac{R}{2}.
 \end{equation}

 Since $f$ has no critical points, the disk $D(s_i, 4|r_i|/R)$ at the pole $s_i$ is mapped univalently  by  the respective branches of $f^{-p_i}$ onto  neighborhoods of the points $b_j =f^{-p_i}(s_i)$. Let $V_j'$ be the component of $f^{-p_j}(U)=f^{-(p_i+1)}(\mathcal{A}_{2R})$ at $s_j$. It is obvious that $\overline{V_j'}\subset V_j$, a component of $f^{-(p_i+1)}(\mathcal{A}_R)$.
 Since $$U_i\subset D(s_i, \frac{2|r_i|}{R})\subset D(s_i, \frac{4|r_i|}{R}), $$  the Koebe distortion theorem implies that for any $z, w\in \widetilde{V}_{ij}$,
 \begin{equation}\label{measW}
 \frac{(f^{p_i})'(z)}{(f^{p_i})'(w)}\leq T(\frac{1}{2}).
 \end{equation}

 By Proposition $7$,  for some $c_1>0$, $|f'(b_{j})|\sim c_1j^{(N-2)/N}$.  Since  $f$ is univalent on the orbit of $\lambda_i$,   there exists  $c_2>0$ such that $|(f^{p_i})'(b_j)|\sim c_2 j^{(N-2)/N}$ and thus
 $$|(f^{p_i})'(z)|> c_2T^{-1}(\frac{1}{2}) j^{\frac{N-2}{N}}, \text{ for any }z\in V'_{j}.$$
   Since $U\subset D(s_i, 2|r_i|/R)$, this implies
 $$V'_{j}\subset D(b_{{j}}, \frac{2T(\frac{1}{2})|r_i|}{c_2|j|^{\frac{N-2}{N}}R}).$$
 This, combined  with equation (\ref{deriatinfinity}) also implies that for all $z\in V'_{j}$,
  $$\sigma_l'(z)\geq \frac{c_2R|j|^{\frac{N-2}{N}}}{2T(1/2)} $$
  and thus completes the proof.
   \end{proof}

\subsection{Proof of Theorem~\ref{main1}}

Let $E=\{z, \sigma^n_l(z)\to \infty, n\to \infty\}$; then
$$
I_l(f)=\cup_{n=0}^\infty f^{-n}(E).
$$
To prove theorem~\ref{main1}, it suffices to show that the measure of the set $E$ is zero.  We assume not and obtain a contradiction.  Let $z_0$ be a Lebesgue density point of the set $E$, and let $z_n=\sigma_l^n(z_0)$. As $z_n\to \infty$, without loss of generality, we may assume that for each $n$, $|z_{n+1}|\geq |z_{n}|\geq R$.   Set $\cala_{r,s} = \{ z, \, | \, r < |z| <s  \} $.

 Since by hypothesis, $K<N$, if follows that  $\cup_{K+1}^N T_i\not= \emptyset$ but $(\cup_{K+1}^N T_i)\cap \sigma_l^{-1}(\cala_R)=\emptyset$.  Therefore, for any $s>r>R$,  there is a
 $\tau>0$ such that
 $$\frac{m(\cala_{r,s}\cap \sigma_l^{-1}(\cala_{r,s}))}{m(A_{r,s})}<1-\tau. $$

Note that if $K=N$, the asymptotic tracts fill up $\sigma_l^{-1}(\cala_R)$.   The proof of non-ergodicity for $K=N=2$ in~\cite{CJK} uses this fact and lends  support to conjecture~\ref{nec}.

The proof of theorem~\ref{main1}, is in two parts depending on the orbit of $z_0$.\\
{\bf Part 1:} Assume that for all $n$, $z_n\in \cup_{i=1}^K T_i$.  This part of the proof depends on  lemma~\ref{lem3a} and   uses  the notation in that lemma.

As in the lemma, for $z_n \in T_i$, set $w_n=\log z_n \in U_{im}$ and    $r_n=\Re w_n$.  Then $F_{im}^{-1}: H_R \to U_{im}$  is the inverse branch
  such that $F_{im}^{-1}(w_n)=w_{n-1}$. The function $F_{im}^{-1}$
    is univalent in the disk $D(w_n, r_n-\log  R)$.  By   lemma~\ref{lem3a},  it follows that  $$|(F_{im}^{-1})'(w_n)|\leq \frac{4\pi}{r_n-\log R}.$$

Note that  it may be that $z_{n-1} \in T_j$, $i \neq j$, $i,j \in S_l$, and similarly,  $w_{n-1}$ may be in a different $U_{jm}$.   For the sake of readability, we will ignore these indices and write $U$ for whichever $U_{im}$ is meant and write $F^{-1}$ for whichever inverse branch is meant.

Next, consider the disk $D(w_n, r_n/4)$. First note that since  $U $ does not intersect the any of the  preimages, $  f^{-1}(T_i)$, $i=K+1, \ldots,N $, there exists a $\tau'>0$, such that
$$\frac{m(D(w_n,\frac{r_n}{4})\cap U)}{m(D(w_n,\frac{r_n}{4})}<1-\tau'.$$
Moreover,  for $w\in D(w_n,r_n/4)$,  the  Koebe distortion theorem implies that
$$|F^{-1}(w)-F^{-1}(w_n)|\leq \frac{4\pi}{r_n-\log R}\cdot \frac{\eta(r_n-\log R)}{(1-\eta)^2},$$where $$\eta=\frac{r_n}{4(r_n-\log R))}<\frac{1}{2}.$$
Therefore, $$F^{-1}(D(w_n, \frac{r_n}{4}))\subset D(w_{n-1}, d) \text{ where } d=8\pi.$$

For each $1\leq k\leq n-1$,  $F^{-1}$ is univalent in the disk $D(w_k, 2d)$ and $(F^{-1})'(w_k)\leq 1/8$. For each $k$,  $1\leq k\leq n-1$,  the Koebe $1/4$ theorem implies that
  $$F^{-1}(D(w_k, d))\subset D(w_{k-1},\frac{d}{2}). $$
Next, iterate $F^{-1}$ and set
 $B_n=F^{-n}(D(w_n, r_n/4))\subset D(w_0, 2^{-m+1}d)$.
 Now since the iterated function   $F^{-n}$ is univalent  on $D(w_n, \Re w_n-\log R)$,   apply Koebe distortion again to get
$$D(w_0, t \rho_n)\subset B_n\subset D(w_0, \rho_n)$$
where $t$ is independent of $n$, and $\rho_n$ is the radius of the smallest disk
centered at $w_0$ containing $B_n$. It follows that $\rho_n\leq 2^{-n+1}d$ which, in turn, implies that $\rho_n\to 0$ as $n\to \infty$.

Part (ii) of the Koebe Distortion Theorem applied to  $F^{-n}$ implies there exists a $\tau''$ such that
  $$\frac{m(B_n\cap E)}{m(E)}\leq 1-T^{-2}(\frac{1}{2})\tau''$$ for all $n$.  In other words, the Lebesgue density of the point $w_0$ is less than $1$,  which contradicts the assumption that $w_0$ is a density point.

\medskip

\noindent{\bf Part 2:} Now consider a subsequence  $z_{n_k} \in \cup_{i\in S_l}( \cup_j V_{j}) $.  
 Let $W=\sigma_l^{-1}(\mathcal{A}_{2R})$ and let  $V_{j}$ be the component of $W$ such that $z_{n_k}\in V'_{j}\subset V_{j}$ for some $n_k$. There is at least one asymptotic tract $T_k$ such that $W \cap T_k =\emptyset$; thus there exists a  $\tau'''>0$ such that
  $$\frac{m(\mathcal{A}_{2R}\cap W)}{m(\mathcal{A}_{2R})}<1-\tau'''.$$
By  lemma \ref{lem3b},
  $$V'_{j}\subset D(b_{j}, \frac{b}{R|j|^{\frac{N-2}{N}}}), $$
 and for any $z\in V'_{j}$,
  $$ \sigma_l'(z)\geq BR|j|^{\frac{N-2}{N}} \text{  and   }  \frac{(f^{p_i})'(z)}{(f^{p_i})'(w)}\leq T(\frac{1}{2}).$$

Therefore, $$\frac{m(V'_{j}\cap \sigma_l^{-1}(W))}{m(V'_{j})}<1-T(\frac{1}{2})^{-2}\tau'''.$$

Without loss of generality,  assume that $|z_{n+1}|\geq |z_n|>>R$ for all $n$.   Then the above inequality, together with lemmas \ref{lem3a} and \ref{lem3b}, show
 $|\sigma_l'(z_n)|>M>1$ for all $n$.

 Let $B_{n_k}=\sigma_l^{-n_k}(V'_{j})$. Then
$$B_{n_k}\subset D(z_0, M^{-n_k}\frac{b}{R|j|^{\frac{N-2}{N}}}).$$
Since $\sigma_l^{-n_k}$ is univalent on $V_{j}\supset V'_{j}$, this implies
$$D(z_0, t \rho_{n_k})\subset B_{n_j-1}\subset D(z_0, \rho_k)$$
where $t$ is independent of $n_j$, and $\rho_{n_k}$ is the radius of the smallest disk
centered at $z_0$ containing $B_{n_j-1}$. It follows that $\rho_{n_k}\to 0$ as $n_k\to \infty$.

Applying part (ii) of the Koebe Distortion Theorem,
   $$\frac{m(B_{n_k}\cap E)}{m(E)}\leq 1-T^{-4}(\frac{1}{2})\tau'''$$ for all $n_k$ which implies that the Lebesgue density of the point $z_0$ is less than $1$.  This  contradicts the assumption that $z_0$ is a density point and completes the proof of theorem~\ref{main1}.

\subsection{Proof of Theorem~\ref{lem2}}

\begin{proof} To prove that $m(L)=0$, note first that
by assumption $\Omega=\cup_{i=K+1}^N \omega(\la_i)$ is a finite union of compact repellers,  so is again a compact repeller.  This implies that the orbits of the non-prepole asymptotic values do not accumulate on  $\Omega$, but actually land on it.  The proof in \cite{KK} assumes $K=0$.  Although this proof is similar to that proof, here we modify it  to take account  of the prepole asymptotic values.

  Let $\calk_{\epsilon}=\{z \, | \, dist(z,\Omega)< \epsilon \}$.  We claim there is an $\epsilon>0$ and an integer $M>0$ such that if
    $y= \la_i$, $i=K+1,\ldots, N$,  $n>M$ and $f^n(y) \in \calk_{\epsilon/2}$, then $f^n(y)\in \Omega$.

 If $ \Omega$ is finite, the claim is obviously true, so assume it is not.     By the compactness assumption, there are no prepoles in $\Omega$ and there
there are constants $\kappa>1$ and $\epsilon >0$   such that  $|(f^n) '(w)| \geq \kappa$ for some $m$ and all $w \in \calk_{\epsilon}$,  and thus for all   $w \in \overline{\calk_{\epsilon/2}}$.
  By the forward invariance of $\Omega$ and this expansion property, $\overline{\calk_{\epsilon/2}} \subset f^n(\overline{\calk_{\epsilon/2})}$.  Let $g$ be the inverse branch of $f^n$ reversing this inclusion.  Then set
 $$A_0=\overline{\calk_{\frac{\epsilon}{2}}} - g(\overline{\calk_{\frac{\epsilon}{2}}}) \text{ and } A_{n+1}=g^n(A_0), n \to \infty.$$
 These disjoint annuli are nested and, since the inverse branches are univalent,  have the same moduli.   Therefore,  if  for some $n$,  $f^n( y) \in \overline\calk_{\epsilon} \setminus  \Omega$, by compactness, there are subsequences of its iterates that converge both to  points in $\overline{A_0}$ and to  points in $ \Omega$.  This is a contradiction because these sets are disjoint and the claim is proved.

 \medskip
 Choose $\epsilon$ as above and set $$\call=\cap_{n\geq 0} f^{-n}(\calk_{\frac{\epsilon}{2}}).$$
 A point $z \in \call \setminus \Omega$ if its full forward orbit belongs to $\calk_{\epsilon/2}$.   We will show $m(\call \setminus  \Omega)=0$.  Since $L \subset \cup_{n=0}^{\infty}f^{-n}(\call)$ and $ \Omega$ is countable, this will imply $m(L)=0$.

 Suppose  $m(\call \setminus  \Omega)>0$ and let $z_0$ be a density point of  $\call \setminus  \Omega$.  Since $\Omega$ is compact, a subsequence $z_k=f^{n_k}(z_0)$   converges to a point $y_0 \in \overline{\calk_{\epsilon/2}} \subset \calk_{\epsilon}$.  Denote the respective inverse branches by $g_k$.  Set $D_k=D(z_k, \epsilon/4)$; then $D_k \subset \calk_{\epsilon}$ and $g_k$ is univalent on $D_k$.   Applying Koebe distortion we obtain
 $$\frac{m(g_k(D_k) \cap \call)}{m(g_k(D_k))} \rightarrow 1 \, \, \text{ and } \frac{m(D_k \cap f^{n_k}(\call))}{m(g_k(D_k))} \rightarrow 1.$$
 Finally, let $U$ be an open set with compact closure contained in $\CC \setminus \overline{\calk}$.   Since the Julia set is the whole sphere, there is an integer $M$ such that $f^M(D_k) \supset \overline{U}$ so that $m(f^{M+n_k}( \call \cap U)) >0$.   For all $k \in \NN$, however, $f^k(\call) \subset \calk_{\epsilon/2}$ so that $f^{M+n_k}( \call \cap U) = \emptyset$. This contradiction shows $m(L)=0$.
 \end{proof}

   \subsection{Proof of the main theorem}

  \begin{thm}\label{main2}  If $f$ is a Nevanlinna function with $1 \leq K < N$ prepole\footnote{ If infinity is an asymptotic value, we consider it a ``prepole of order 0''.}  asymptotic values and  $N-K$  asymptotic values that  accumulate on a compact repeller, then $f$ acts ergodically on its Julia set.
  \end{thm}

\begin{proof}
  Let $A$ be an $f$-invariant subset of the Julia set with positive measure. We will show that $A=\widehat{\mathbb{C}}$ up to  a set of measure zero.
 Let $z_0$ be a Lebesgue density  point of $A$ and denote its orbit by $z_n=f^n(z)$, $ n=0, 1, \ldots .$  We proved above that the measures of each of   the sets $I$,  $I_l$  and $L$ is zero.
 Since these three sets together contain all points whose orbits accumulate on $\cup_{i=1}^N \omega(\lambda_i)$ we assume that $z_0$ is not among them.

 \medskip

 By the above, the density point $z_0$ of $A$ has an accumulation point $y \in \CC \setminus \cup_{i=1}^N \omega(\la_i)$.  Recall that $\Omega =  \cup_{i=K+1}^N \omega(\la_i)$. Hence there is an $\epsilon>0$ so that  $2\epsilon=dist(y,  \Omega)>0$.  Thus, there is a
 subsequence $\{n_{j}\}$  in $\NN$ such that $z_{n_j}\to y$
as $j\to \infty$ and $dist(z_{n_i}, \Omega) \geq \epsilon$.

Let $B_j=B(z_{n_j}, \epsilon)$ and $V_j=B(z_{n_j}, \epsilon/2)$.
Let $g_j$ be the  inverse branch  of $f^{-n_j}$ that sends $z_{n_j}$ to $z$; it  is a  univalent function on $B_{j}$.  Let $U_j=g_j(V_j)$.
All of the  inverse branches of $f$
 are contracting with respect to the hyperbolic metric on $\CC\setminus  \Omega$; this implies that   $g_j'\to 0$ on $V_j$ as $j$ goes to $\infty$, which in turn
  implies that the diameter of $U_j$ tends to $0$. Since  $g_j$ is univalent on $B_j$,  the Koebe distortion  theorem shows that
  that  $U_j$ is almost a disk. Since $z$ is a density point of $A$,
$$
\lim_{j\to \infty} \frac{m(A\cap U_j)}{m(U_j)}=1.
$$
Applying Koebe distortion again, since $A$ is an invariant subset, we get
$$
\lim_{j\to \infty} \frac{m(A\cap V_j)}{m(V_j)}=1.
$$

This and the fact that $V_j$ approaches $B_y=B(y, \epsilon/2)$ as $j$ goes to $\infty$ together  imply that $B_y\subset A$ up to a set  measure zero. Since $A$ is  an invariant subset of the Julia set $J$, $f^n(B_{y})\subset A \subset J$. Since $A$ is in the Julia set,  $B_{y}$ is also in the Julia set and $f^{n}(B_{0})$ approaches  $\CC$ as $n$ goes to $\infty$.  Therefore,  $A=\widehat{\mathbb{C}}$ up to a zero-measure set and the proof of Theorem~\ref{main2} is complete.
\end{proof}

\medskip
\begin{remark}  The assumption that the $\omega$-limit sets of the non pre-polar asymptotic values are compact repellers says the Julia set is the whole sphere and gives us the expansion we need to prove our theorem.   We could replace this by assuming that the Julia set is the sphere and that the orbits of the non-prepolar orbits are bounded.  Then the main theorem of~\cite{GKS} implies the expansion we need exists.
\end{remark}

\medskip
\begin{remark} Another application of the results in ~\cite{GKS} and~\cite[Theorem 1.1]{RVS} to the Nevanlinna functions $f$ of our main theorem is that $f$ supports  no invariant line field.
\end{remark}

\medskip
\begin{remark}  Finally, the results in~\cite{KU} applied to the Nevanlinna functions of our main theorem prove that $f$ has a $\sigma$-finite ergodic conservative $f$-invariant  measure  absolutely continuous  with respect to the Lebesgue measure.
 \end{remark}


%
%

\end{document}